\documentclass[onecolumn]{IEEEtran}

\usepackage{amsmath,amsfonts,amssymb,euscript, graphicx, 
epsfig,
enumerate,float,afterpage, subfigure, ifthen}%

\usepackage{url}
\usepackage{hyperref}

\newtheorem{thm}{Theorem}

\newtheorem{lem}{Lemma}
\newtheorem{defn}{Definition}
\newtheorem{assumption}{Assumption}

\newcommand{\expect}[1]{\mathbb{E}\left[#1\right]}

\newcommand{\norm}[1]{\left|\left|{#1}\right|\right|}

\newcommand{\script}[1]{{{\cal{#1} }}}

\begin{document}

\title
  {Online Convex Optimization with Time-Varying Constraints}
\author{Michael J. Neely , Hao Yu \\ University of Southern California\\ \url{http://www-bcf.usc.edu/~mjneely/}\\
}

\markboth{}{Neely}

\maketitle

\begin{abstract} 
This paper considers online convex optimization with time-varying constraint functions. Specifically, we have a sequence of convex objective functions $\{f_t(x)\}_{t=0}^{\infty}$ and convex constraint functions $\{g_{t,i}(x)\}_{t=0}^{\infty}$ for $i \in \{1, ..., k\}$.  The functions are gradually revealed 
over time. For a given $\epsilon>0$, the goal is to choose points $x_t$ every step $t$, without knowing the $f_t$ and $g_{t,i}$ functions on that step, to achieve a time average at most $\epsilon$ worse than the best fixed-decision that could be chosen with hindsight, subject to the time average of the 
constraint functions being nonpositive. It is known that this goal is generally impossible.   This paper develops an online algorithm that solves the problem with $O(1/\epsilon^2)$ convergence time in the special case when all constraint functions are nonpositive over a 
common subset of $\mathbb{R}^n$.  Similar performance is shown in an expected sense 
when the common subset assumption is removed but the 
constraint functions are assumed to vary according to a random process that is  independent and identically distributed (i.i.d.) over time slots $t \in \{0, 1, 2, \ldots\}$. Finally, 
in the special case when both the constraint and objective functions are i.i.d. over time slots $t$, the algorithm is 
shown to come within $\epsilon$ of optimality with respect to the best (possibly time-varying)  causal  policy that knows the full probability distribution. 
\end{abstract}

\section{Introduction}

Fix $n$ and $k$ as positive integers. Let $\script{X} \subseteq \mathbb{R}^n$ be a convex and compact set.  
Let $\{f_t\}_{t=0}^{\infty}$  be a sequence of continuous and convex \emph{objective functions} defined
over $x \in \script{X}$. 
For each $i \in \{1, \ldots, k\}$, let  $\{g_{t,i}\}_{t=0}^{\infty}$  be a sequence of convex \emph{constraint functions} defined 
over $x \in \script{X}$. The $f_t$ and $g_{t,i}$ functions are initially unknown. They are gradually revealed over time.  Every slot $t \in \{0, 1, 2, \ldots\}$, a controller chooses a (possibly random) vector $X_t \in \script{X}$, without knowledge of the $f_t$ and $g_{t,i}$ functions.  This incurs a cost $f_t(X_t)$ and generates a collection of penalties $g_{t,1}(X_t), g_{t,2}(X_t), \ldots, g_{t,k}(X_t)$.  
The functions $f_t$ and $g_{t,i}$ are revealed at the end of slot $t$, after the $X_t$ decision is made.  

While the functions are unknown, they are assumed to have bounded subgradients. 
The algorithm of this paper makes use of subgradient information that is revealed at the end of 
every slot.  This formulation is inspired by the classic \emph{online convex optimization} 
framework of Zinkevich \cite{online-convex}.  Specifically, work in \cite{online-convex} considers the case of pure objective function minimization, so that there are no constraint functions $g_{t,i}$.  It develops the online subgradient projection algorithm: 
$$ x_{t} = \mathcal{P}_{\script{X}}[x_{t-1} -  \epsilon f'_{t-1}(x_{t-1})] $$
where $\delta>0$ is a fixed step-size, $f'_{t-1}(x_{t-1})$ denotes a subgradient of $f_{t-1}$ at the point $x_{t-1}$, and  $\mathcal{P}_{\script{X}}$ denotes projection onto the set $\script{X}$.  
 This algorithm makes a decision at time $t$ without knowledge of the $f_t$ function and using only the subgradient 
 information of the function $f_{t-1}$ at the previously chosen point $x_{t-1}$.  Remarkably, this algorithm ensures the following holds for all points $x^* \in \script{X}$: 
 \begin{align}
 \frac{1}{T}\sum_{t=0}^{T-1} f_t(x_t) \leq \frac{1}{T}\sum_{t=0}^{T-1}f_t(x^*) + c \epsilon , \quad \forall T \geq 1/\epsilon^2 \label{eq:zink1} 
 \end{align} 
 where $c$ is a system constant.  In particular, for each positive integer $T$, 
 one can define $x^*_T = \arg\min_{x \in \script{X}} \sum_{t=0}^{T-1} f_t(x)$ as the best \emph{fixed decision} at time $T$ (defined with full knowledge of the future).  Thus, \eqref{eq:zink1} ensures the  
 Zinkevich algorithm (which does not know the future) achieves a time average objective value that is at most $O(\epsilon)$ worse that the best fixed decision, for all $T \geq 1/\epsilon^2$.  We call this an \emph{$O(\epsilon)$-approximation with convergence time $1/\epsilon^2$.}   Work in \cite{kale-universal} presents a simple example with linear objective functions to show that $O(1/\epsilon^2)$ is the best possible asymptotic convergence time for \emph{any} algorithm, while 
 improved convergence times are possible under more stringent 
 \emph{strongly convex} properties of $f_t$. 
 
The bound \eqref{eq:zink1} is stated in terms of \emph{convergence time} to an $O(\epsilon)$-approximation.   
This is closely related to the metric of $O(\sqrt{T})$ \emph{regret}.  The $O(\sqrt{T})$ regret metric requires, for all $x^* \in \script{X}$: 
\begin{equation} \label{eq:regret} 
\sum_{t=0}^{T-1} f_t(x_t) \leq \sum_{t=0}^{T-1} f_t(x^*) + O(\sqrt{T}), \quad \forall T>0 
\end{equation} 
Fix $\epsilon>0$. Dividing \eqref{eq:regret} by $T$ shows that an algorithm that achieves $O(\sqrt{T})$ regret also achieves an $O(\epsilon)$-approximation with convergence time $1/\epsilon^2$.  Strictly speaking, 
the $O(\sqrt{T})$ regret property is stronger than \eqref{eq:zink1} and Zinkevich achieves it in \cite{online-convex} 
by using a \emph{diminishing step size} $\epsilon_t = 1/\sqrt{t}$, rather than a fixed step size $\epsilon$.\footnote{Note that fixed-stepsize algorithms are often preferred because they are more adaptive to emerging conditions and have \emph{dynamic regret} properties,  as discussed in \cite{online-convex}, where regret guarantees hold over arbitrary subintervals of time.} 

  Alternatively, a
standard doubling trick (as in \cite{Shalev-Shwartz11FoundationTrends}) can often be used to convert 
algorithms that achieve \eqref{eq:zink1} for arbitrary values of $\epsilon>0$ into 
 an algorithm to achieve $O(\sqrt{T})$ regret for all $T>0$.  This is done by implementing the former algorithm over successive intervals of time, where each interval $m$ is twice the size of its predecessor and uses a value $\epsilon_m$ that is held fixed over the interval but decreases at interval boundaries.  For simplicity of exposition, this paper focuses on the convergence time definition of \eqref{eq:zink1}.  The doubling method is described 
 in Appendix \ref{appendix:regret} to show how the algorithm of this paper can also be 
 modified to achieve $O(\sqrt{T})$ regret.

 \subsection{Prior work with constraints} 
 
 One wonders if a similar result can be obtained for extended problems with convex constraint 
 functions $g_{t,i}$.  The answer is ``no.'' Specifically, work in \cite{constraint-online-impossible} presents a simple example 
 of a problem with a time-varying convex objective function $f_t(x)$ and a single time-varying constraint
 function $g_t(x)$.   The goal is to minimize the time average of $f_t(x_t)$ subject to the time average of $g_t(x_t)$ being
 less than or equal to 0.  It is assumed that the constraint is feasible.  
  In this context, an $\epsilon$-approximation for the first $T$ slots 
  is defined by requiring the time average constraint 
  to be violated by at most $\epsilon$, while the objective time average must be at most 
  $\epsilon$ larger than that
  achieved by the best \emph{constraint-achieving fixed-action policy} over these slots.  The example in \cite{constraint-online-impossible} constructs a sample path such that, 
  for a given $\epsilon>0$, no algorithm 
  can produce an $\epsilon$-approximation.  Intuitively, the sample path is constructed so that 
  any algorithm that achieves 
  $\epsilon$-optimality at a particular time $T$ necessarily makes decisions over the first $T$ steps that lead to significant 
  constraint violations at time  $2T$.  On the other hand, alternative  actions would allow the constraints to be satisfied at time $2T$, but
  would necessarily produce  a huge deviation from objective optimality at time $T$. 
   
  Work in \cite{mahdavi-learning} considers constrained 
  online convex optimization in the special case when the constraint functions $g_i(x)$, $i \in \{1, ..., k\}$,  
  are known and 
  do not depend on time.  This scenario can be solved by the classical Zinkevich algorithm by 
  defining a modified convex set: 
  $$\hat{\script{X}} = \script{X} \cap \{x \in \script{X} : g_i(x) \leq 0 \quad \forall i \in \{1, ..., k\}\}$$ 
  However, when the number of constraints $k$ is large, the set $\hat{\script{X}}$ can be complicated and the 
   projection operation in \eqref{eq:zink1} can be difficult to implement.   In contrast, the set $\script{X}$ might be a hypercube for which projections are easy.  The work in \cite{mahdavi-learning} develops an algorithm that 
   uses projections onto the simpler set $\script{X}$, so that per-slot complexity is smaller, 
   while achieving an $O(\epsilon)$-approximation with convergence time $O(1/\epsilon^{4})$. Specifically,  
   for all $x^* \in \hat{\script{X}}$ the algorithm ensures: 
   \begin{align}
&\frac{1}{T}\sum_{t=0}^{T-1} f_t(x_t) \leq \frac{1}{T}\sum_{t=0}^{T-1}f_t(x^*) + O(\epsilon) , \quad \forall T \geq 1/\epsilon^3
 \label{eq:mad1} \\
 &\frac{1}{T}\sum_{t=0}^{T-1} g_{i}(x_t) \leq  O(\epsilon) , \quad  \forall i \in \{1, ..., k\}, \: \forall T \geq 1/\epsilon^4  \label{eq:mad2} 
 \end{align}
 The convergence time here is $1/\epsilon^4$, which is not as good as the $1/\epsilon^2$ convergence time of Zinkevich. However,  the algorithm is simpler to implement on every slot.\footnote{Notice that, for online problems, \emph{convergence time} and \emph{algorithm complexity} are two different things.  Work in \cite{kale-universal} 
 shows an example problem for which any algorithm that achieves an $\epsilon$-approximation must have convergence time $\Omega(1/\epsilon^2)$, even if that algorithm uses an ``infinitely fast'' computer that can solve arbitrarily complex problems on every slot.  This is because only one sample is revealed per slot, and at least $\Omega(1/\epsilon^2)$ 
 samples of the system must be observed.}  Recent followup work in \cite{jenatton-learning} shows a tradeoff between the convergence times in 
 \eqref{eq:mad1} and \eqref{eq:mad2} can be achieved by using time varying step sizes: The convergence time for \eqref{eq:mad1} can be changed to $O(1/\epsilon^{\max\{\frac{1}{1-\delta}, \frac{1}{\delta}\}})$, for any $\delta \in (0,1)$, with a  corresponding convergence time tradeoff in \eqref{eq:mad2} 
 of $O(1/\epsilon^{\frac{2}{\delta}})$.  Recent work in \cite{hao-fast-online-learning} introduces a new technique to reduce the convergence time of \eqref{eq:mad1} to 
 $O(1/\epsilon^2)$ with convergence time of the constraints \eqref{eq:mad2} being $O(1/\epsilon)$. 
    
A related problem formulation in \cite{online-constraints-batch} treats time-varying constraints in the special case when both the objective and constraint functions vary according to an i.i.d. stochastic process.   In such cases, the ``worst-case'' sample paths of the work in \cite{constraint-online-impossible} occur with probability zero, and so it may be possible to construct online algorithms that allow $\epsilon$-approximations for both the objective and constraints in an expected sense or probabilistic sense.  For this scenario, \cite{online-constraints-batch} develops a batch algorithm, rather than an online algorithm.  Specifically, it shows that after observing the system for a 
sufficiently long time, one can produce a random variable that solves a related constrained optimization problem. 
To our knowledge, the problem of constructing an online algorithm for the case of i.i.d. problems has not 
be solved before.

 
 \subsection{Our contributions} 
 
 The current paper first considers general sample paths with no probabilistic model. It 
 develops an $O(\epsilon)$ approximation with convergence time $1/\epsilon^2$ 
 for the case when the time-varying constraint functions satisfy a deterministic Slater condition and  
 optimality is restricted by a \emph{common subset} assumption.  These conditions hold, for example, 
 when the constraint functions do not vary with time, as studied 
 in  \cite{mahdavi-learning}\cite{jenatton-learning}\cite{hao-fast-online-learning}, and our new algorithm 
improves on the convergence time of \cite{mahdavi-learning}\cite{jenatton-learning} and matches that of 
\cite{hao-fast-online-learning} in that special case.  

Next, we consider a stochastic model and assume that the vector-valued constraint function $(g_{t,1}, \ldots, g_{t,k})$ 
is chosen by nature according to a process that is independent and identically distributed (i.i.d.) over time slots $t$, 
while the objective function $f_t$ varies according to 
an arbitrary stochastic process.  Under this probabilistic structure, 
we show our algorithm again provides an $O(\epsilon)$-approximation (in an expected sense) 
with convergence time $1/\epsilon^2$, but optimality is now measured against the more general class of 
fixed-decision policies that achieve the desired constraints in expectation.   Finally, we consider the 
case when the vector-valued function $(f_t, g_{t,1}, \ldots, g_{t,k})$ (which includes the constraint functions \emph{and} the objective function) is i.i.d. over slots.  In this case, the algorithm achieves similar convergence time, but this time optimality is with respect to any (possibly time-varying)
 policy that does not know the future.    This i.i.d. case is similar to that treated in \cite{online-constraints-batch}
 using a non-online batch algorithm. 
 To our knowledge, the current paper provides the 
 first online algorithm with convergence guarantees for such problems.  
 In this i.i.d. case, the stringent deterministic Slater condition is replaced by a mild  
 Lagrange multiplier assumption, and the algorithm does not require knowledge of the Lagrange multipliers.

Our algorithm can be implemented every slot as a projection of a certain vector onto the set $\script{X}$, as 
in the Zinkevich algorithm \eqref{eq:zink1}.  However, while Zinkevich uses only the subgradient information $f'_t(X_t)$ at the end of each step $t$, 
our algorithm uses the subgradient information $f'_t(X_t), g_{t,1}'(X_t), ..., g_{t,k}'(X_t)$ as well as the function values $g_{t,1}(X_t), ..., g_{t,k}(X_t)$. 

The $O(1/\epsilon^2)$ convergence time achieved in this paper meets two fundamental lower 
bounds known to hold for special unconstrained versions of our problem.  First, it matches the $\Omega(1/\epsilon^2)$ lower bound known to hold for unconstrained 
online convex optimization.  Indeed,  \cite{kale-universal}  presents a simple system with linear $f_t$ functions that vary i.i.d. over slots such that all algorithms 
have convergence time at least $\Omega(1/\epsilon^2)$. Second, 
since our objective and constraint functions are not required to be smooth, a special case of our i.i.d. formulation is when there is no time variation and the problem reduces to minimizing a  possibly nonsmooth convex function  subject to possibly nonsmooth inequality constraints. 
Nesterov shows in \cite{nesterov-book} that, for such nonsmooth convex minimization problems, all 
algorithms that make decisions
based on  linear combinations of subgradients have convergence time at least $\Omega(1/\epsilon^2)$, even if there are no inequality constraints.  The primal-dual subgradient method, also known as the Arrow-Hurwicz-Uzawa subgradient method, can minimize a possibly nonsmooth convex function subject to possibly nonsmooth inequality constraints with $O(1/\epsilon^2)$ convergence \cite{Nedic09_PrimalDualSubgradient}. However, its implementation requires an upper bound on the optimal Lagrangian multipliers, which is typically unknown in practice.  In contrast, the algorithm of this paper 
does not require knowledge of the Lagrange multipliers. 

\section{Formulation} \label{section:formulation} 

Let $\{f_t\}_{t=0}^{\infty}$ and $\{g_{t,i}\}_{t=0}^{\infty}$, $i \in \{1, \ldots, k\}$ 
be continuous and convex functions defined over a convex decision set $\script{X} \subseteq \mathbb{R}^n$. 
The functions are possibly nonsmooth and are not required to be differentiable.

\subsection{Boundedness and subgradient assumptions} 

The decision set $\script{X}$ is convex and compact.  Let $D \geq 0$ be a constant that represents 
the \emph{diameter} of $\script{X}$, so that 
\begin{equation} \label{eq:D} 
\norm{x-y} \leq D \quad \forall x, y \in \script{X} 
\end{equation} 
where $\norm{z} = \sqrt{\sum_{i=1}^n z_i^2}$ denotes the standard Euclidean norm.  Let $F \geq 0$ be a constant that represents a universal bound on all function values, so that for all $i \in \{1, \ldots, k\}$, all $t \in \{0, 1, 2, \ldots\}$, and all $x \in \script{X}$ we have: 
\begin{align}
|f_t(x)| \leq F \: \: , \: \: 
|g_{t,i}(x)| \leq F  \label{eq:F} 
\end{align} 

The functions $f_{t}$ and $g_{t,i}$ are assumed to have \emph{subgradients} for all $x \in \script{X}$.  Let  
$f_t'(x)$ and $g_{t,i}'(x)$ denote  \emph{particular} subgradient vectors defined at $x \in \script{X}$.   
By definition of a subgradient, the $f_t'(x)$ and $g_{t,i}'(x)$ vectors satisfy
\begin{align}
f_t(y) &\geq f_t(x) + f_t'(x)^T(y-x) \quad \forall x, y \in \script{X}\label{eq:subgradient-f}  \\
g_{t,i}(y) &\geq g_{t,i}(x) + g_{t,i}'(x)^T(y-x) \quad \forall x, y \in \script{X} \label{eq:subgradient-g} 
\end{align}
where $z^T$ denotes the transpose of a (column) vector $z$, and
\eqref{eq:subgradient-g} 
holds for all $i \in \{1, \ldots, k\}$. 
In the special case when $f_t$ and $g_{t,i}$ are differentiable at a point $x \in \script{X}$ 
then $f_t'(x) = \nabla f_t(x)$ and $g_{t,i}'(x) = \nabla g_{t,i}(x)$.  Assume the subgradients are bounded so that there is a positive constant $G$ such that for all  $i \in \{1, \ldots, k\}$, all  
$t \in \{0, 1, 2, \ldots\}$, and all $x \in \script{X}$ we have:
\begin{align}
\norm{f_t'(x)} \leq G \: \: , \: \: 
\norm{g_{t,i}'(x)} \leq G  \label{eq:G} 
\end{align}

\subsection{Optimization over a common subset} 

For Sections \ref{section:formulation}-\ref{section:bounds}, no probabilistic assumptions are made concerning
the time-varying functions $f_t$ and $g_{t,i}$.  However, 
the $g_{t,i}$ functions are assumed to satisfy the following \emph{Slater condition}.

\begin{assumption} \label{assumption:slater} (Slater condition) There exists a real number $\eta >0$ and 
a vector $s \in \script{X}$, called the \emph{Slater vector},  
such that for all slots $t \in\{0, 1, 2, \ldots\}$ we have: 
\begin{eqnarray}
g_{t,i}(s) \leq -\eta \quad, \forall i \in \{1, \ldots, k\}  \label{eq:slater} 
\end{eqnarray}
\end{assumption} 

This assumption is natural in many cases, including cases when there is an all-zero decision that
allocates zero power or zero resources so that all constraints are loose. It is also natural in the special
case when the constraint functions have no time variation, so that $g_{t,i} = g_i$ (as treated in \cite{mahdavi-learning}\cite{jenatton-learning}\cite{hao-fast-online-learning}). 
Notice that the Slater condition places no restrictions on the time-varying objective functions $f_t$. 

Define the \emph{common subset} $\script{A}$ as the set of all $x \in \script{X}$ such that: 
\begin{equation} \label{eq:set-A}
g_{t,i}(x) \leq 0 \quad , \forall i \in \{1, \ldots, k\}, \forall t \in \{0, 1, 2, \ldots\} 
\end{equation} 
The Slater condition implies that $s \in \script{A}$ and hence $\script{A}$ is nonempty.  The
set $\script{A}$ represents the common subset of $\script{X}$ over which all constraint functions are nonpositive 
for all time.  It can be shown that  $\script{A}$
is a compact set. Indeed, this follows because $\script{X}$ is compact, functions $g_{t,i}$ are continuous over $\script{X}$, and $\script{A}\subseteq\script{X}$. 

We shall construct an online algorithm and compare its performance against the performance
of all fixed-decisions in the common subset $\script{A}$.  Specifically, fix $\epsilon>0$ and let $T$ be a positive integer. 
An algorithm for making decisions 
$X_t \in \script{X}$ over slots $t \in \{0, 1, 2, \ldots\}$ is said to be an \emph{$\epsilon$-approximation}  
with  \emph{convergence time $T$} if the following holds for all slots $t \geq T$: 
\begin{align*}
 \frac{1}{t}\sum_{\tau=0}^{t-1} f_t(X_t) &\leq \frac{1}{t}\sum_{\tau=0}^{t-1} f_t(x) + \epsilon \quad, \forall x \in \script{A} \\
 \frac{1}{t}\sum_{\tau=0}^{t-1} g_{t,i}(X_t) &\leq \epsilon \quad \forall i \in \{1, \ldots, k\} 
 \end{align*}
The above inequalities imply that, after a transient time of size $T$, 
the algorithm comes within $\epsilon$ of satisfying all constraints and also 
achieves an objective value that is within  $\epsilon$ of that of the best fixed-decision 
vector in the set $\script{A}$.  An algorithm
is said to be an $O(\epsilon)$ approximation with convergence time $T$ if the above 
inequalities hold with $\epsilon$ replaced by $c \epsilon$ for some 
constant $c$ that does not depend on $\epsilon$ and $T$.
The set $\script{A}$ may be complex and may not be known to the system 
controller.  The algorithm presented in the next section 
 does not require knowledge of set $\script{A}$. 

The above goal compares against fixed-decision vectors $x \in \script{A}$ that make 
all constraint functions nonpositive for all slots $t$. 
It is more desirable to optimize over the larger set of all fixed-decision vectors $x \in \script{X}$
that are only required to satisfy the time average constraints in the limit as $t\rightarrow\infty$ (using a $\limsup$ if limits do not exist).  
However, work in \cite{constraint-online-impossible} shows this more ambitious goal is generally impossible to achieve.  
Fortunately, Section \ref{section:probability} shows that the \emph{same} algorithm of this paper achieves 
this more ambitious goal in an \emph{expected sense} and  in the special case when a probability model is introduced and
the vector-valued sequence of constraint functions $\{(g_{t,1}, \ldots, g_{t,k})\}_{t=0}^{\infty}$ is i.i.d. over slots $t \in\{0, 1, 2, \ldots\}$. 
Stronger optimization results are achieved in Section \ref{section:model2} under the additional assumption that
the objective functions are also i.i.d. over slots.

\section{Algorithm} \label{section:algorithm} 

Fix $x_0$ as any point in $\script{X}$.  Define the initial decision as $X_{0} = x_0$.  
 For each constraint $i \in \{1, \ldots, k\}$ let $Q_i(t)$ represent a \emph{virtual queue} 
 defined over slots $t \in \{0, 1, 2, 3, \ldots\}$ 
 with initial conditions $Q_i(0)=Q_i(1)=0$ and update equation: 
\begin{equation} \label{eq:q-update} 
Q_i(t+1) = \max[Q_i(t) + g_{t-1,i}(X_{t-1}) + g_{t-1,i}'(X_{t-1})^T(X_t-X_{t-1}), 0] \quad \forall t \in \{1, 2, 3, \ldots\} 
\end{equation} 
Our algorithm uses parameters $V>0, \alpha>0$ and makes decisions as follows:  

\begin{itemize} 
\item On slot $t=0$, choose $X_0 =x_0$. 
\item At the start of each slot $t \in \{1, 2, 3, \ldots\}$, observe $Q_i(t)$ for all $i \in \{1, \ldots, k\}$ and choose $X_t \in \script{X}$ to minimize the following expression: 
\begin{equation} \label{eq:decision} 
 \left[Vf_{t-1}'(X_{t-1})^T + \sum_{i=1}^k Q_i(t)g_{t-1,i}'(X_{t-1})^T\right] X_t + \alpha \norm{X_t-X_{t-1}}^2
 \end{equation} 
where the historical values $X_{t-1}$, $f_{t-1}'(X_{t-1})$, and $g_{t-1,i}'(X_{t-1})$ for $i \in \{1, \ldots, k\}$ are treated as fixed and known constants in the above expression. 
\item At the end of each slot $t \in \{1, 2, 3, \ldots\}$, update virtual queues $Q_i(t)$ for all $i \in \{1, \ldots, k\}$ via \eqref{eq:q-update}. 
\end{itemize} 

The decision $X_t$ is chosen on slot $t$ to minimize the expression \eqref{eq:decision}  
over all options in the set $\script{X}$.  
The next lemma shows that this minimization can be implemented by a simple \emph{projection} onto the 
set $\script{X}$.   

\subsection{Implementation as a projection} 

For each vector $y \in \mathbb{R}^n$ define the projection operator $\script{P}_{\script{X}}[y]$ as: 
$$ \script{P}_{\script{X}}[y] = \arg\inf_{x \in \script{X}} \norm{x-y}^2 $$
Since $\script{X}$ is a compact and convex set, this projection always exists and is unique. 

\begin{lem} (Projection implementation) Fix $\alpha>0$ and $t\in\{1, 2, 3, \ldots\}$. The unique $X_t \in \script{X}$ that minimizes \eqref{eq:decision} is: 
\begin{equation} \label{eq:projection} 
X_t = \script{P}_{\script{X}} \left[ X_{t-1} + \frac{W_t}{2\alpha}   \right] 
\end{equation} 
where $W_t$ is defined from the historical information as follows: 
\begin{align*}
W_t = Vf'_{t-1}(X_{t-1}) + \sum_{i=1}^k Q_i(t)g_{t-1,i}'(X_{t-1}) 
\end{align*}  
\end{lem} 
\begin{proof}  The proof is similar to a proof given in \cite{hao-constrained-learning}.  The expression \eqref{eq:decision} is equal to 
$W_t^TX_t + \alpha \norm{X_t-X_{t-1}}^2$. Hence, 
\begin{align*}
X_t &= \arg\inf_{x \in \script{X}} \left[ W_t^Tx+ \alpha \norm{x-X_{t-1}}^2  \right]\\
&\overset{(a)}{=}\arg\inf_{x \in \script{X}} \left[ \frac{\norm{W_t}^2}{4\alpha}+ W_t^T(x-X_{t-1}) + \alpha \norm{x-X_{t-1}}^2  \right]\\
&\overset{(b)}{=}\arg\inf_{x \in \script{X}} \norm{\sqrt{\alpha}(x-X_{t-1}) - \frac{W_t}{2\sqrt{\alpha}}}^2\\
&\overset{(c)}{=}\arg\inf_{x \in \script{X}} \norm{x- \left(X_{t-1} + \frac{W_t}{2\alpha}\right)}^2 \\
&=  \script{P}_{\script{X}} \left[ X_{t-1} + \frac{W_t}{2\alpha}   \right] 
\end{align*}
where (a) holds because the minimizer does not change when the constant $\norm{W_t}^2/(4\alpha)-W_t^TX_{t-1}$
is added to the expression; (b) holds by expanding the square; (c) holds because the minimizer is unchanged when the expression is divided  by $\sqrt{\alpha}$. 
\end{proof}

\subsection{Virtual queue analysis} 

The following lemma provides a collection of bounds on the virtual queue values, one bound for each real number $\beta>0$. 

\begin{lem} \label{lem:virtual-queues}  Fix $i \in \{1, \ldots, k\}$ and $T \in \{1, 2, 3, \ldots\}$. 
Under the virtual queue update equation \eqref{eq:q-update} the following holds for all real numbers $\beta>0$: 
\begin{equation} \label{eq:vq} 
\frac{1}{T}\sum_{t=0}^{T-1} g_{t,i}(X_t) \leq \frac{Q_i(T+1)}{T}  + \frac{G^2}{4\beta}  + \frac{\beta}{T}\sum_{t=1}^{T} \norm{X_t-X_{t-1}}^2
\end{equation} 
\end{lem}

\begin{proof}  
Since $\max[x,0] \geq x$, the update equation \eqref{eq:q-update} implies for each slot $t\in \{1, 2, 3, \ldots\}$: 
$$ Q_i(t+1) \geq Q_i(t) + g_{t-1,i}(X_{t-1}) + g_{t-1,i}'(X_{t-1})^T(X_t-X_{t-1}) $$
Rearranging terms gives: 
\begin{align}
Q_i(t+1) -Q_i(t) &\geq g_{t-1,i}(X_{t-1}) +  g_{t-1,i}'(X_{t-1})^T(X_t-X_{t-1})    \nonumber \\
&\overset{(a)}{\geq} g_{t-1,i}(X_{t-1}) - G\norm{X_t-X_{t-1}} \nonumber\\
&= g_{t-1,i}(X_{t-1}) - \frac{G^2}{4\beta}  - \beta \norm{X_t-X_{t-1}}^2 + \left(\frac{G}{2\sqrt{\beta}}-\sqrt{\beta}\norm{X_t-X_{t-1}}\right)^2  \nonumber \\
&\geq g_{t-1,i}(X_{t-1}) - \frac{G^2}{4\beta}  - \beta  \norm{X_t-X_{t-1}}^2 \label{eq:to-sum} 
\end{align}
where (a) holds by the Cauchy-Schwartz inequality and the fact that $\norm{g_{t-1,i}'(X_{t-1})}\leq G$. 
Fix $T\geq 1$. Summing \eqref{eq:to-sum} over $t \in \{1, \ldots, T\}$ gives: 
$$ Q_i(T+1)-Q_i(1) \geq \sum_{t=0}^{T-1}g_{t,i}(X_t) - \frac{TG^2}{4\beta} - \beta\sum_{t=1}^{T}\norm{X_t-X_{t-1}}^2 $$
Dividing by $T$, rearranging terms, and using $Q_i(1)=0$ yields the result. 
\end{proof} 

The above lemma shows that it is desirable to maintain a low value of the virtual queues $Q_i(t)$ and  also 
maintain a low value of $\norm{X_t-X_{t-1}}^2$. 

\subsection{Intuition on algorithm construction} 

Define $Q(t) = (Q_1(t), \ldots, Q_k(t))$ as the vector of virtual queues for slot $t \in \{1, 2,3, \ldots\}$. 
Define $L(t) = \frac{1}{2}\norm{Q(t)}^2$.  The function $L(t)$ is a scalar measure of the virtual queue vector and shall be called a \emph{Lyapunov function}. Define $\Delta(t) = L(t+1)-L(t)$.   The intuition behind the algorithm is that it makes  decisions $X_t \in \script{X}$ every slot $t$ to minimize a bound on the expression: 
$$ \underbrace{\Delta(t)}_{\mbox{drift}} + \underbrace{\alpha \norm{X_t-X_{t-1}}^2 + Vf_{t-1}'(X_{t-1})^T(X_t-X_{t-1})}_{\mbox{weighted penalty}} $$
The term $\Delta(t)$ can be viewed as a \emph{Lyapunov drift} term: Making this term small intuitively 
helps to maintain small values of the virtual queues.  The remaining term can be viewed as a weighted \emph{penalty term}.  The penalty term includes $\norm{X_t-X_{t-1}}^2$ because Lemma \ref{lem:virtual-queues} shows it is desirable for this to be small.  Intuitively, 
the expression $f_{t-1}'(X_{t-1})^T(X_t-X_{t-1})$ appears in the penalty because we have learned from the original Zinkevich algorithm (for unconstrained problems) that it is desirable for this term to be small.  The next subsections compute bounds on the above drift-plus-penalty expression.  The weights $\alpha$ and $V$ shall be chosen carefully to establish desirable performance. 

\subsection{Sample path drift analysis} 

\begin{lem} For all slots $t \in \{1, 2, 3, \ldots\}$ we have: 
\begin{equation} \label{eq:drift} 
\Delta(t) \leq B + \sum_{i=1}^k Q_i(t)[g_{t-1,i}(X_{t-1}) + g_{t-1,i}'(X_{t-1})^T(X_t-X_{t-1})] 
\end{equation} 
where $\Delta(t) = \frac{1}{2}\norm{Q(t+1)}^2-\frac{1}{2}\norm{Q(t)}^2$, and the constant $B$ is defined: 
\begin{equation} \label{eq:B} 
B = \frac{k(F+GD)^2}{2}
\end{equation} 
where constants $D, F, G$ are defined in \eqref{eq:D}, \eqref{eq:F}, \eqref{eq:G}. 
\end{lem} 
\begin{proof} 
Fix $i \in \{1, \ldots, k\}$. Using the fact that $\max[z,0]^2 \leq z^2$ for all $z \in \mathbb{R}$ in the update equation 
\eqref{eq:q-update} gives: 
\begin{align*}
Q_i(t+1)^2 &\leq \left[Q_i(t) + g_{t-1,i}(X_{t-1}) + g_{t-1,i}'(X_{t-1})^T(X_t-X_{t-1})\right]^2 \\
&= Q_i(t)^2 + \left(g_{t-1,i}(X_{t-1}) + g_{t-1,i}'(X_{t-1})^T(X_t-X_{t-1})\right)^2  \\
& \quad + 2Q_i(t)\left[g_{t-1,i}(X_{t-1}) + g_{t-1,i}'(X_{t-1})^T(X_t-X_{t-1})\right] 
\end{align*}
Define $b_i(t) = g_{t-1,i}(X_{t-1}) + g_{t-1,i}'(X_{t-1})^T(X_t-X_{t-1})$.  Rearranging the above inequality gives
$$ Q_i(t+1)^2 - Q_i(t)^2 \leq b_i(t)^2 + 2Q_i(t)\left[g_{t-1,i}(X_{t-1}) + g_{t-1,i}'(X_{t-1})^T(X_t-X_{t-1})\right] $$
Summing over all $i \in \{1, \ldots, k\}$ and dividing by $2$ yields: 
$$ \Delta(t) \leq \frac{1}{2}\sum_{i=1}^k b_i(t)^2 + \sum_{i=1}^k Q_i(t)\left[g_{t-1,i}(X_{t-1}) + g_{t-1,i}'(X_{t-1})^T(X_t-X_{t-1})\right] $$
The result follows because $|b_i(t)| \leq F + GD$ for all $i$ and all $t$.  
\end{proof} 

Adding the penalty term $Vf_{t-1}'(X_{t-1})^T(X_t-X_{t-1}) + \alpha \norm{X_t-X_{t-1}}^2$ to both sides 
of \eqref{eq:drift} gives a bound on the drift-plus-penalty expression: 
\begin{align}
&\Delta(t) + Vf_{t-1}'(X_{t-1})^T(X_t-X_{t-1}) + \alpha \norm{X_t-X_{t-1}}^2 \nonumber \\
& \leq B + Vf_{t-1}'(X_{t-1})^T(X_{t} -X_{t-1})  + \alpha \norm{X_t-X_{t-1}}^2\nonumber  \\
& \quad + \sum_{i=1}^k Q_i(t)[g_{t-1,i}(X_{t-1}) + g_{t-1,i}'(X_{t-1})^T(X_t-X_{t-1})]   \label{eq:transparent} 
\end{align}
The algorithm decision in \eqref{eq:decision}  is now transparent: \emph{The algorithm chooses $X_t \in \script{X}$ to minimize the right-hand-side of \eqref{eq:transparent}}.

\subsection{Strong convexity analysis} 

Recall that $\script{X}$ is a convex subset of $\mathbb{R}^n$. 
Fix a real number $c>0$. 
A function $h:\script{X}\rightarrow\mathbb{R}$ is said to be 
\emph{$c$-strongly convex} if $h(x) - \frac{c}{2}\norm{x}^2$ is convex over $x \in \script{X}$. It is easy to see that if  
$q:\script{X}\rightarrow\mathbb{R}$ is a convex function, then for any constant $c>0$ and any 
vector $b \in \mathbb{R}^n$, the function $q(x) + \frac{c}{2}\norm{x-b}^2$ is $c$-strongly convex.  Further, it is known 
that if $h:\script{X}\rightarrow\mathbb{R}$ is a $c$-strongly convex function that is minimized at a point $x^{min} \in \script{X}$, then (see, for example, \cite{hao-fast-convex-SIAM}): 
\begin{equation} 
h(x^{min}) \leq h(y) - \frac{c}{2}\norm{y-x^{min}}^2 \quad \forall y \in \script{X} \label{eq:strongly} 
\end{equation} 

Notice that 
the expression on the right-hand-side of \eqref{eq:transparent} is a $(2\alpha)$-strongly convex function of $X_t$ (due to the $\alpha \norm{X_t-X_{t-1}}^2$ term).  Since $X_t$  minimizes  this strongly convex expression over all vectors in $\script{X}$, it follows from \eqref{eq:strongly}  that for all vectors $y \in \script{X}$: 
\begin{align}
&\Delta(t) + Vf_{t-1}'(X_{t-1})^T(X_t-X_{t-1}) + \alpha \norm{X_t-X_{t-1}}^2 \nonumber \\
& \leq B + Vf_{t-1}'(X_{t-1})^T(y -X_{t-1})  + \alpha \norm{y-X_{t-1}}^2\nonumber  \\
& \quad + \sum_{i=1}^k Q_i(t)[g_{t-1,i}(X_{t-1}) + g_{t-1,i}'(X_{t-1})^T(y-X_{t-1})]   - \alpha\norm{y-X_t}^2 \label{eq:y-class} 
\end{align}
This leads to the following lemma.

\begin{lem} (Sample path drift-plus-penalty bound) For every vector $y \in \script{X}$ and every slot $t\in\{1, 2, 3,\ldots\}$ we have 
\begin{align}
&\Delta(t) +  \frac{\alpha}{2} \norm{X_t-X_{t-1}}^2  \nonumber \\
& \quad \leq B + Vf_{t-1}(y) - Vf_{t-1}(X_{t-1})  + \alpha \norm{y-X_{t-1}}^2 - \alpha\norm{y-X_t}^2 \nonumber  \\
& \quad + \sum_{i=1}^k Q_i(t)g_{t-1,i}(y) + \frac{V^2G^2}{2\alpha}  \label{eq:drift-plus-penalty} 
\end{align}
where the constant $B$ is defined in \eqref{eq:B}. 
\end{lem} 
\begin{proof} Fix $t \in \{1, 2, 3, \ldots\}$ and $y \in \script{X}$. 
Using the subgradient inequalities \eqref{eq:subgradient-f}-\eqref{eq:subgradient-g} gives for all $i \in \{1, \ldots, k\}$: 
\begin{align*}
 f_{t-1}'(X_{t-1})^T(y-X_{t-1}) &\leq f_{t-1}(y) - f_{t-1}(X_{t-1})  \\
g_{t-1,i}'(X_{t-1})^T(y-X_{t-1}) &\leq g_{t-1,i}(y) - g_{t-1,i}(X_{t-1}) 
\end{align*}
Substituting these two inequalities into the right-hand-side of \eqref{eq:y-class} gives: 
\begin{align*}
&\Delta(t) + Vf_{t-1}'(X_{t-1})^T(X_t-X_{t-1}) + \alpha \norm{X_t-X_{t-1}}^2 \nonumber \\
& \leq B + Vf_{t-1}(y) - Vf_{t-1}(X_{t-1})  + \alpha \norm{y-X_{t-1}}^2 - \alpha\norm{y-X_t}^2 \nonumber  \\
& \quad + \sum_{i=1}^k Q_i(t)g_{t-1,i}(y) 
\end{align*}
By rearranging the above inequality so that the left-hand-side has the same form 
as \eqref{eq:drift-plus-penalty}, it remains to show that: 
$$ - Vf_{t-1}'(X_{t-1})^T(X_t-X_{t-1}) -\frac{\alpha}{2}\norm{X_t-X_{t-1}}^2  \leq \frac{V^2G^2}{2\alpha}$$
To this end, by completing the square we have: 
\begin{align*}
&-Vf'_{t-1}(X_{t-1})^T(X_t-X_{t-1}) - \frac{\alpha}{2}\norm{X_t-X_{t-1}}^2 \\
&= -\norm{\frac{Vf'_{t-1}(X_{t-1})}{2\sqrt{\alpha/2}} + \sqrt{\alpha/2}(X_t-X_{t-1})}^2 + \frac{V^2}{2\alpha}\norm{f'_{t-1}(X_{t-1})}^2 \\
&\leq \frac{V^2G^2}{2\alpha}
\end{align*}
where we have used the fact that $\norm{f'_{t-1}(X_{t-1})}\leq G$. 
\end{proof}

\section{Performance bounds}  \label{section:bounds} 

This section provides performance bounds on the algorithm of the previous section.  It is assumed throughout that the algorithm uses parameters $\alpha>0, V>0$. 

\subsection{Objective bound}

\begin{thm} \label{thm:performance-bound}  (Objective function bound) Suppose the set 
$\script{A}$ is nonempty. Then the algorithm of this paper ensures the following 
for all integers $T\geq 1$ and all $x \in \script{A}$: 
\begin{equation} \label{eq:objective-function}  
\frac{1}{T}\sum_{t=0}^{T-1}f_t(X_t) \leq \frac{1}{T}\sum_{t=0}^{T-1} f_t(x) +  \frac{B}{V} + \frac{VG^2}{2\alpha} + \frac{\alpha D^2}{VT} 
\end{equation} 
where constants $B, G, D$ are defined in \eqref{eq:B}, \eqref{eq:G}, \eqref{eq:D}. 
In particular, if we fix $\epsilon>0$ and define $V=1/\epsilon$, $\alpha = 1/\epsilon^2$, then for all $T\geq 1/\epsilon^2$ we have: 
$$ \frac{1}{T}\sum_{t=0}^{T-1}f_t(X_t) \leq \frac{1}{T}\sum_{t=0}^{T-1}f_t(x) + O(\epsilon) $$
\end{thm} 
\begin{proof} 
Substituting $y=x$ into \eqref{eq:drift-plus-penalty} and using the fact that $g_{t,i}(x)\leq 0$ for all $i \in \{1, \ldots, k\}$ yields
the following for all $t \in \{1, 2, 3, \ldots\}$: 
$$ \Delta(t) \leq B + Vf_{t-1}(x) - Vf_{t-1}(X_{t-1}) + \alpha\norm{x-X_{t-1}}^2 - \alpha \norm{x-X_t}^2 + \frac{V^2G^2}{2\alpha} $$
where the nonnegative term $\frac{\alpha}{2}\norm{X_t-X_{t-1}}^2$ has been dropped from the left-hand-side of the above inequality.  
Fix $T \geq 1$. Summing over $t \in \{1, \ldots T\}$ and dividing by $T$ gives: 
\begin{align*}
\frac{L(T+1)-L(1)}{T} &\leq B + \frac{V}{T}\sum_{t=0}^{T-1}f_t(x) - \frac{V}{T}\sum_{t=0}^{T-1}f_t(X_t) + \frac{V^2G^2}{2\alpha}+ \frac{\alpha\norm{x-X_0}^2-\norm{x-X_T}^2}{T}\\
& \leq B + \frac{V}{T}\sum_{t=0}^{T-1}f_t(x) - \frac{V}{T}\sum_{t=0}^{T-1}f_t(X_t) + \frac{V^2G^2}{2\alpha}+ \frac{\alpha D^2}{T}
\end{align*}
where we have used the fact that $\norm{x-X_0}^2 \leq D^2$.  Rearranging terms, 
substituting $L(1)=0$, and 
neglecting the nonnegative term $L(T+1)$ gives the result.  
\end{proof}

\subsection{Queue bound}

\begin{lem} \label{lem:slater-drift}  Under the deterministic Slater condition (Assumption \ref{assumption:slater}), 
the following holds 
for all slots $t \in\{1, 2, 3, \ldots\}$: 
\begin{eqnarray} \label{eq:drift-deterministic} 
\Delta(t) \leq B + RV  - \eta \norm{Q(t)} + \alpha\norm{s - X_{t-1}}^2- \alpha\norm{s - X_{t}}^2 
\end{eqnarray}
where $s$ is the Slater vector, $R$ is defined: 
\begin{equation} \label{eq:R} 
R = \frac{VG^2}{2\alpha} + 2F
\end{equation} 
and constants $B, G, F$ are defined in \eqref{eq:B}, \eqref{eq:G}, \eqref{eq:F}.  Note that $R=O(1)$ whenever
$\alpha \geq V$. 
\end{lem}
\begin{proof} The result follows immediately by substituting $y=s$ into  \eqref{eq:drift-plus-penalty} and 
using \eqref{eq:slater} and the facts that: (i)  $Vf_{t-1}(s)-Vf_{t-1}(X_{t-1})\leq 2VF$, (ii) 
$\sum_{i=1}^k Q_i(t) \geq \norm{Q(t)}$. 
\end{proof} 

\begin{thm} \label{thm:deterministic-queue-bound} (Queue bound) Suppose $V$ is a positive integer 
and $Q_i(0)=Q_i(1)=0$ for all $i \in \{1, \ldots, k\}$.  Under the deterministic Slater condition 
we have for all slots $t$: 
$$ \norm{Q(t)} \leq \theta V $$
where constants $\theta$ and $\delta$  are defined: 
\begin{align}
\theta &= \max\left[ \delta, \frac{B+RV}{\eta V} + \frac{\alpha D^2}{\eta V (V+1)}  + \frac{\delta (V+2)}{2V}\right]  \label{eq:theta} \\
\delta &= \sqrt{k}(F+DG) \label{eq:delta} 
\end{align} 
where the constants $B$, $R$, and $D$ are defined in \eqref{eq:B}, \eqref{eq:R}, \eqref{eq:D}. 
In particular, if $\alpha = V^2$ then $\norm{Q(t)} \leq O(V)$ for all slots $t$.
\end{thm} 

\begin{proof} 
The queue equation \eqref{eq:q-update} implies that each queue $Q_i(t)$ changes
by at most $F+DG$ over one slot, and so the norm of the vector $Q(t)$ changes by at most 
$\delta$ over one slot.  
 Then for all slots $t \in \{0, 1, 2, \ldots, V\}$ we have: 
$$ \norm{Q(t)} \leq \delta t  \leq \delta V \leq \theta V $$
and so the desired bound holds for all slots $t \leq V$.  Fix a slot $T\geq V$ and suppose $\norm{Q(t)} \leq \theta V$
for all $t \leq T$.  We show the bound also holds at time $T+1$, that is, we show  $\norm{Q(T+1)}\leq \theta V$. 

\begin{itemize} 
\item Case 1: Suppose $\norm{Q(T+1)} \leq \norm{Q(T-V)}$.  Since $\norm{Q(T-V)}\leq \theta V$, we are done. 

\item Case 2: Suppose $\norm{Q(T+1)}>\norm{Q(T-V)}$. Summing \eqref{eq:drift-deterministic} over  $t \in \{T-V, \ldots, T\}$ gives: 
$$L(T+1) - L(T-V) \leq B(V+1) + RV(V+1) - \eta \sum_{t=T-V}^{T} \norm{Q(t)} + \alpha\norm{s-X_{T-V}}^2 - \alpha\norm{s-X_{T}}^2$$
Using the norm definition of $L(t)$ and the fact that $\norm{s-X_{T-V}}^2 \leq D^2$ (from \eqref{eq:D}) gives: 
$$ \underbrace{\frac{1}{2}\norm{Q(T+1)}^2 - \frac{1}{2}\norm{Q(T-V)}^2}_{\mbox{positive}} \leq B(V+1) + RV(V+1) + \alpha D^2 -\eta\sum_{t=T-V}^{T} \norm{Q(t)} $$ 
The term in the underbrace is positive by the assumption for this Case 2 and so: 
\begin{equation} \label{eq:sumV} 
(B+RV)(V+1)  + \alpha D^2 > \eta \sum_{t=T-V}^T \norm{Q(t)} 
\end{equation} 
Now suppose $\norm{Q(T+1)} > \theta V$ (we shall reach a contradiction). Since the queue norm can change by at most 
$\delta$ every slot, we have: 
$$ \norm{Q(t)} > \theta V - (T+1-t)\delta  \quad , \forall t \in \{0, 1, \ldots, T+1\}$$
Substituting this into \eqref{eq:sumV} gives: 
$$ (B+RV)(V+1) + \alpha D^2 > \eta \sum_{t=T-V}^T [\theta V - (T+1-t)\delta]  = \eta \theta V (V+1) - \eta \delta\frac{(V+1)(V+2)}{2} $$
Rearranging the above inequality gives
$$ \theta  < \frac{B+RV}{\eta V} + \frac{\alpha D^2}{\eta V (V+1)} + \frac{\delta (V+2)}{2V} $$
The above inequality  contradicts the definition of $\theta$ in \eqref{eq:theta}. 
\end{itemize} 
\end{proof} 

\begin{thm} \label{thm:constraint-bound} (Constraint bound) 
Suppose the deterministic Slater condition  holds (Assumption \ref{assumption:slater}). 
Let $V$ be a positive integer and define $\alpha = V^2$. Then for each constraint $i \in \{1, \ldots, k\}$ we have: 
$$ \frac{1}{T}\sum_{t=0}^{T-1} g_{t,i}(X_t) \leq \frac{\theta V}{T} + \frac{G^2}{4V} + \frac{G^2(1+\theta \sqrt{k})^2}{4V}  \quad, \forall T>0$$
where $k$ is the number of constraints and constants $\theta$, $G$ are defined in \eqref{eq:theta}, \eqref{eq:G}. 
Hence, constraint violations are $O(1/V)$ whenever $T \geq V^2$. 
\end{thm} 
\begin{proof} 
Fix a constraint $i \in \{1, \ldots, k\}$ and fix an integer $T>0$.  Lemma \ref{lem:virtual-queues} (with $\beta = V$) implies: 
\begin{equation} \label{eq:almost-done} 
\frac{1}{T}\sum_{t=0}^{T-1} g_{t,i}(X_t) \leq \underbrace{\frac{Q_i(T+1)}{T}}_{\leq \frac{\theta V}{T}} + \frac{G^2}{4V}  + \frac{V}{T}\sum_{t=1}^T\norm{X_t-X_{t-1}}^2
\end{equation} 
where the term marked by the first underbrace is at most $\theta V/T$ by Theorem \ref{thm:deterministic-queue-bound}. 

It remains to bound the final term on the right-hand-side of \eqref{eq:almost-done}. To this end, 
fix $t \in \{1, 2, \ldots, T\}$. The decision $X_t \in \script{X}$ made at time $t$ minimizes the strongly convex
expression \eqref{eq:decision} over all other
vectors in $\script{X}$.  Since $X_{t-1} \in \script{X}$ we have by the strong convex minimization fact \eqref{eq:strongly}:
\begin{align*}
&\left[Vf_{t-1}'(X_{t-1})^T  + \sum_{i=1}^k Q_i(t) g_{t-1,i}'(X_{t-1})\right]X_t  + \alpha \norm{X_t - X_{t-1}}^2  \\
&\leq \left[Vf_{t-1}'(X_{t-1})^T  +  \sum_{i=1}^k Q_i(t) g_{t-1,i}'(X_{t-1})^T\right]X_{t-1} + 0 - \alpha \norm{X_t-X_{t-1}}^2
\end{align*}
Rearranging terms gives: 
\begin{align*}
2\alpha \norm{X_t-X_{t-1}}^2 &\leq \left[Vf_{t-1}'(X_{t-1})^T + \sum_{i=1}^kQ_i(t)g_{t-1,i}'(X_{t-1})^T\right](X_{t-1}-X_t)  \\
&\overset{(a)}{\leq} VG\norm{X_{t-1} - X_t} + \sqrt{k}G\norm{Q(t)}\norm{X_{t-1}-X_t} \\
&\overset{(b)}{\leq} VG\norm{X_{t-1} - X_t}  + \sqrt{k}G\theta V \norm{X_{t-1}-X_t}
\end{align*}
where (a) uses the Cauchy-Schwartz inequality and the fact that $\sum_{i=1}^k Q_i(t) \leq \sqrt{k}\norm{Q(t)}$; (b) uses Theorem \ref{thm:deterministic-queue-bound}.  Hence: 
$$ \norm{X_{t-1}-X_t} \leq \frac{VG(1+ \theta\sqrt{k})}{2\alpha}$$
It follows that: 
$$ \frac{V}{T}\sum_{t=1}^T \norm{X_t-X_{t-1}}^2 \leq V \left[\frac{VG(1+ \theta\sqrt{k})}{2\alpha}\right]^2 = \frac{G^2(1+\theta\sqrt{k})^2}{4V} $$
\end{proof} 

The interpretation of Theorem \ref{thm:performance-bound} and Theorem \ref{thm:constraint-bound} is that for any 
$\epsilon>0$,  one can select parameters $V$ and $\alpha$ so that the resulting algorithm achieves an $O(\epsilon)$ approximation
with convergence time $T=1/\epsilon^2$.  This is done by selecting $V$ as the smallest integer greater than or equal to $1/\epsilon$, and 
selecting $\alpha = V^2$. 

\section{Stochastic analysis} \label{section:probability} 

This section develops stronger performance guarantees when a probabilistic structure is imposed. 
\subsection{Probability model} 

Let $\Omega$ be a finite dimensional vector space. Let $\{\omega_t\}_{t=0}^{\infty}$ and $\{\eta_t\}_{t=0}^{\infty}$ be two (possibly dependent) sequences  of random vectors in 
$\Omega$.  The functions $f_t$ and $g_{t,i}$ are determined by these processes on each slot $t \in \{0, 1, 2, \ldots\}$ by: 
\begin{align}
f_t(x) &= \hat{f}(x,\eta_t) \label{eq:fhat}   \\
g_{t,i}(x) &= \hat{g}_i(x,\omega_t) \label{eq:ghat} 
\end{align}
where functions $\hat{f}$ and $\hat{g}_i$ are continuous and convex  with respect to  $x \in \script{X}$.  The $\hat{f}$ and $\hat{g}_i$ 
functions  are bounded and have bounded subgradients with respect to $x \in \script{X}$, so the resulting $f_t$ and $g_{t,i}$ functions are indeed continuous, convex, and satisfy the bounds \eqref{eq:F} and \eqref{eq:G} for all sots $t$.

\begin{itemize} 
\item Model 1:  Our first model 
assumes   $\{\omega_t\}_{t=0}^{\infty}$ is independent and identically distributed (i.i.d.) over slots $t$, 
but the sequence $\{\eta_t\}_{t=0}^{\infty}$ is arbitrary and can depend on the former sequence.  This means that 
the vector-valued constraint functions $\{(g_{t,1}, \ldots, g_{t,k})\}_{t=0}^{\infty}$ are i.i.d. over slots $t$, while the functions
$\{f_t\}_{t=0}^{\infty}$ have 
arbitrary time-variation.  

\item Model 2: Our second model assumes $\eta_t=\omega_t$ for all slots $t \in \{0, 1, 2, \ldots\}$ 
and again assumes $\{\omega_t\}_{t=0}^{\infty}$ is i.i.d. over
slots.  Thus, the sequence of vector-valued functions $\{(f_t, g_{t,1}, \ldots, g_{t,k})\}_{t=0}^{\infty}$ 
is i.i.d. over slots $t$.  In particular, this model
requires both the constraint and objective functions to have i.i.d. time variation. 
\end{itemize} 

Regardless of the model, it is assumed that the system controller has no knowledge of the probability distribution of the 
random sequences.  Let $\{X_t\}_{t=0}^{\infty}$ be the (possibly random) sequence of control decisions. 
On each slot $t$, the $X_t$ decision is assumed to be a random variable that takes values in the convex set $\script{X}$,
and the resulting function values 
$f_t(X_t)$, $g_{t,1}(X_t), \ldots, g_{t,k}(X_t)$ are assumed to be random variables with well defined expectations (the 
expectations are finite because of \eqref{eq:F}). 
This section and the next assume Model 1 holds.  Model 2 is a special case of Model 1 and allows for stronger results
that are presented in Section  \ref{section:model2}.

\subsection{Goal for Model 1} 

Let $\tilde{\script{A}}$ be the set of all vectors $x \in \script{X}$ such that:
\begin{equation} \label{eq:A-constraint} 
\expect{g_{t,i}(x)} \leq 0 \quad, \forall i \in \{1, \ldots, k\}, \forall t \in \{0, 1, 2, \ldots\} 
\end{equation} 
The set $\tilde{\script{A}}$ always contains the set  $\script{A}$ defined in \eqref{eq:set-A}, so that
$\script{A} \subseteq \tilde{\script{A}} \subseteq \script{X} \subset \mathbb{R}^n$. 
In particular, if the Slater condition (Assumption \ref{assumption:slater}) holds, then both $\script{A}$ and $\tilde{\script{A}}$ are nonempty. 
The set $\tilde{\script{A}}$ can be shown to be compact.\footnote{The set $\tilde{\script{A}}$ must be bounded because it is a subset of the compact (and hence bounded) set $\script{X}$.  To show $\tilde{\script{A}}$ is closed,  note that 
$\hat{g}_i(y, \omega) - G\norm{x-y} \leq \hat{g}_i(x,\omega) \leq \hat{g}_i(y,\omega) + G\norm{x-y}$ for all $x,y \in \script{X}$, $\omega \in \Omega$, $i \in \{1, \ldots, k\}$. Taking expectations shows that $\expect{g_{t,i}(x)}$ is
a continuous function of $x \in \script{X}$ for all $i \in \{1, \ldots, k\}$.}   Suppose $\tilde{\script{A}}$ is nonempty and fix $x \in \tilde{\script{A}}$. 
By the law of large numbers we have for all $i \in \{1, \ldots, k\}$. 
$$ \lim_{t\rightarrow\infty} \frac{1}{t}\sum_{\tau=0}^{t-1} g_{\tau,i}(x) \leq 0 \quad (\mbox{with prob 1})$$

It is useful to redefine an $\epsilon$-approximation using expectations. Fix $\epsilon>0$ and let $T$ be a positive integer.  
An algorithm for making decisions 
$X_t \in \script{X}$ over slots $t \in \{0, 1, 2, \ldots\}$ is said to be an \emph{$\epsilon$-approximation in the expected sense} 
with  \emph{convergence time $T$} if the following holds for all slots $t \geq T$: 
\begin{align*}
 \frac{1}{t}\sum_{\tau=0}^{t-1} \expect{f_t(X_t)} &\leq \frac{1}{t}\sum_{\tau=0}^{t-1} \expect{f_t(x)} + \epsilon \quad, \forall x \in \tilde{\script{A}} \\
 \frac{1}{t}\sum_{\tau=0}^{t-1} \expect{g_{t,i}(X_t)} &\leq \epsilon \quad \forall i \in \{1, \ldots, k\} 
 \end{align*}
 
 \subsection{Performance under Model 1} 
 
 \begin{lem} \label{lem:iid-lemma} Consider Model 1 and assume $\tilde{\script{A}}$ is nonempty. For every vector 
 $x \in \tilde{\script{A}}$ and every slot $t \in \{1, 2, 3, \ldots\}$ we have: 
 \begin{align}
&\expect{\Delta(t)} +  \frac{\alpha}{2} \expect{\norm{X_t-X_{t-1}}^2}  \nonumber \\
& \quad \leq C + V\expect{f_{t-1}(x)} - V\expect{f_{t-1}(X_{t-1})}  + \alpha \expect{\norm{x-X_{t-1}}^2}- \alpha\expect{\norm{x-X_t}^2} + \frac{V^2G^2}{2\alpha}  \label{eq:iid-drift-plus-penalty} 
\end{align}
where
\begin{equation} \label{eq:C}
C = B + kG(F+DG)
\end{equation} 
\end{lem} 

\begin{proof} 
Fix $x \in \tilde{\script{A}}$ and fix $t \in \{1, 2, 3, \ldots\}$. Substituting $y=x$ into \eqref{eq:drift-plus-penalty} gives 
\begin{align}
&\Delta(t) +  \frac{\alpha}{2} \norm{X_t-X_{t-1}}^2  \nonumber \\
& \quad \leq B+ Vf_{t-1}(x) - Vf_{t-1}(X_{t-1})  + \alpha \norm{x-X_{t-1}}^2 - \alpha\norm{x-X_t}^2 \nonumber  \\
& \quad + \sum_{i=1}^k Q_i(t)g_{t-1,i}(x) + \frac{V^2G^2}{2\alpha} \label{eq:to-sub}  
\end{align}
The queue update equation \eqref{eq:q-update} ensures 
$|Q_i(t) -Q_i(t-1)|\leq F+DG$ for each $i \in \{1, \ldots, k\}$, and so: 
\begin{align*} 
\sum_{i=1}^k[Q_i(t)-Q_i(t-1)]g_{t-1,i}(x) &\leq \underbrace{kG(F+DG)}_{C-B} 
\end{align*}
Substituting the above inequality into the right-hand-side of \eqref{eq:to-sub} gives: 
\begin{align*}
&\Delta(t) +  \frac{\alpha}{2} \norm{X_t-X_{t-1}}^2  \nonumber \\
& \quad \leq C+ Vf_{t-1}(x) - Vf_{t-1}(X_{t-1})  + \alpha \norm{x-X_{t-1}}^2 - \alpha\norm{x-X_t}^2 \nonumber  \\
& \quad + \sum_{i=1}^k Q_i(t-1)g_{t-1,i}(x) + \frac{V^2G^2}{2\alpha} 
\end{align*}
Taking expectations of both sides gives: 
\begin{align*}
&\expect{\Delta(t)} +  \frac{\alpha}{2} \expect{\norm{X_t-X_{t-1}}^2}  \nonumber \\
& \quad \leq C + V\expect{f_{t-1}(x)} - V\expect{f_{t-1}(X_{t-1})}  + \alpha \expect{\norm{x-X_{t-1}}^2} - \alpha\expect{\norm{x-X_t}^2} \nonumber  \\
& \quad + \sum_{i=1}^k \expect{Q_i(t-1)}\expect{g_{t-1,i}(x)} + \frac{V^2G^2}{2\alpha} 
\end{align*}
where we have used the fact that $\omega_{t-1}$ is independent of $Q_i(t-1)$  to break the 
expectation of $Q_i(t-1)g_{t-1,i}(x)$ into a product of expectations.   The result follows by noting 
that for all $i \in \{1, \ldots, k\}$ we have $\expect{Q_i(t-1)}\geq 0$ (since virtual queues are nonnegative) 
and $\expect{g_{t-1,i}(x)}\leq 0$ (by \eqref{eq:A-constraint}). 
\end{proof}  
 
 \begin{thm} \label{thm:model-1-performance} (Performance under Model 1)  Consider Model 1 and assume $\tilde{\script{A}}$ is nonempty.  Let $V$ be a positive integer 
 and define $\alpha = V^2$.  With these parameters, the algorithm satisfies the following. 
 
 a) For every vector $x \in \tilde{\script{A}}$ and for all positive integers $T>0$ we have: 
  \begin{equation} \label{eq:objective-function-model1}  
\frac{1}{T}\sum_{t=0}^{T-1}\expect{f_t(X_t)} \leq \frac{1}{T}\sum_{t=0}^{T-1}\expect{f_t(x)} +  \frac{C}{V} + \frac{G^2}{2V} + \frac{VD^2}{T} 
\end{equation}
where constants $C, G, D$ are defined in \eqref{eq:C}, \eqref{eq:G}, \eqref{eq:D}.  
Hence, if $T\geq V^2$, the time average expected objective   is within $O(1/V)$ of that of the optimal fixed-decision $x \in \tilde{\script{A}}$.

b) If the Slater condition (Assumption \ref{assumption:slater}) holds, then for all $T>0$ and all $i \in \{1, \ldots, k\}$ we have: 
 $$ \frac{1}{T}\sum_{t=0}^{T-1} g_{t,i}(X_t) \leq \frac{\theta V}{T} + \frac{G^2}{4V} + \frac{G^2(\theta+1)^2}{V}  $$
 and so the expected constraints satisfy the same inequality. Hence, if $T\geq V^2$ then 
 the constraints are within $O(1/V)$ of being satisfied.
 \end{thm} 
 
 The interpretation of this theorem is that for any $\epsilon>0$, one can choose $V = \lceil1/\epsilon\rceil$ and define $\alpha = V^2$.  The resulting algorithm achieves an $O(\epsilon)$-approximation in the expected sense with convergence time $V^2$.  Notice that part (b) of the above
 theorem is identical to Theorem \ref{thm:constraint-bound}.  It suffices to prove part (a). 
 
\begin{proof} (Theorem \ref{thm:model-1-performance} part (a)) 
Rearranging terms in \eqref{eq:iid-drift-plus-penalty}, substituting $\alpha =V^2$, 
and neglecting the nonnegative term $\alpha \expect{\norm{X_t-X_{t-1}}^2}$ gives the following for all $t \in \{1, 2, 3, \ldots\}$: 
$$ \expect{\Delta(t)} +  V\expect{f_{t-1}(X_{t-1})}  \leq V\expect{f_t(x)} + C + \frac{G^2}{2}  + V^2\expect{\norm{x-X_{t-1}}^2} - V^2 \expect{\norm{x-X_t}^2} $$ 
Fix $T>1$. Summing over $t \in \{1, \ldots T\}$ and dividing by $T$ gives: 
\begin{align*}
\frac{\expect{L(T+1)}-\expect{L(1)}}{T} + \frac{V}{T}\sum_{t=0}^{T-1}\expect{f_t(X_t)} &\leq \frac{V}{T}\sum_{t=0}^{T-1}\expect{f_t(x)} + C + \frac{G^2}{2}+ \frac{V^2\expect{\norm{x-X_0}^2}-V^2\expect{\norm{x-X_T}^2}}{T}\\
& \leq \frac{V}{T}\sum_{t=0}^{T-1}\expect{f_{t}(x)} + C +  \frac{G^2}{2}  + \frac{V^2 D^2}{T} 
\end{align*}
where we have used the fact that $\norm{x-X_0}^2 \leq D^2$ with probability 1.  Dividing by $V$, using $\expect{L(1)}=0$, and 
neglecting the nonnegative term $\expect{L(T+1)}$ gives \eqref{eq:objective-function-model1}.
\end{proof}  

\section{Stochastic analysis for Model 2} \label{section:model2} 

Recall that Model 2 assumes the sequence of vector-valued 
functions $\{(f_t, g_{t,1}, \ldots, g_{t,k})\}_{t=0}^{\infty}$ is i.i.d. over slots.  This allows much stronger
results to be obtained. Specifically: 
\begin{itemize} 
\item We shall remove the need for the Slater condition, so that Assumption \ref{assumption:slater} is no longer needed.
Instead, we replace this assumption with a mild Lagrange multiplier assumption.
\item Rather than simply comparing our algorithm to the best fixed-decision policy that meets the constraints, 
we shall compare with all alternative causal policies, including all fixed-decision policies as well as all time-varying 
policies that make decisions based on full knowledge of the underlying probability distributions. 
This requires optimality over such polices to be characterized.  This is done in the 
next subsection via a concept of \emph{valid decision sequences}. 
\end{itemize} 

\subsection{Optimality and Lagrange multipliers} 

Let $\{X_t\}_{t=0}^{\infty}$ be a sequence of random vectors, each vector taking values in the decision set $\script{X}$. 
For each slot $t \in \{0, 1, 2, \ldots\}$, define $\script{H}(t)$ as the \emph{history} up to but not including slot $t$. Specifically, for each slot $t>0$ we have
$$\script{H}(t) = (\omega_0, \omega_1, \ldots, \omega_{t-1}; X_0, X_1, \ldots, X_{t-1}) $$
 The history is defined to be null at $t=0$, so that $\script{H}(0)=0$.  
 We want to consider decision sequences that are \emph{causal}.  Specifically,  for each $t \in \{0, 1, 2, \ldots\}$, 
 the vector $X_t$ should be  chosen as   a deterministic or random function of the history $\script{H}(t)$, with no knowledge of the future.  
 Since $\{\omega_t\}_{t=0}^{\infty}$ is i.i.d. over slots, a causal decision  
 should have the property that $X_t$ is independent of $\omega_t$ for each slot $t \in \{0, 1, 2, \ldots\}$. This motivates the following definition. 
 
 \begin{defn}  A sequence of random vectors $\{X_t\}_{t=0}^{\infty}$ is a \emph{valid decision sequence} 
 if the following hold for all slots $t \in \{0, 1, 2, \ldots\}$: 
 \begin{itemize} 
 \item $X_t \in \script{X}$.
 \item $X_t$ is independent of $\omega_t$. 
 \end{itemize} 
 \end{defn} 

The goal is to make make valid decisions $\{X_t\}_{t=0}^{\infty}$ over time to solve: 
\begin{align}
\mbox{Minimize:} \quad & \limsup_{T\rightarrow\infty} \frac{1}{T}\sum_{t=0}^{T-1} \expect{f_t(X_t)} \label{eq:p1} \\
\mbox{Subject to:} \quad & \limsup_{T\rightarrow\infty} \frac{1}{T}\sum_{t=0}^{T-1} \expect{g_{t,i}(X_t)} \leq 0 \quad , \forall i \in \{1, \ldots, k\} \label{eq:p2} \\
& X_t \in \script{X} \quad , \forall t \in \{0, 1, 2, \ldots\}  \label{eq:p3} 
\end{align}
where expectations are taken with respect to the random functions and the possibly random decisions. The
problem \eqref{eq:p1}-\eqref{eq:p3} is said to be \emph{feasible} if there exists a valid decision sequence
that satisfies the constraints \eqref{eq:p2}-\eqref{eq:p3}. Assume the problem is feasible.  
Let $f^*$ denote the 
infimum objective value \eqref{eq:p1} over all valid decision sequences that satisfy 
the constraints \eqref{eq:p2}-\eqref{eq:p3}. 


Define $h:\script{X}\rightarrow\mathbb{R}^{k+1}$ by: 
$$ h(x) = (\expect{f_t(x)} , \expect{g_{t,1}(x)}, \ldots, \expect{g_{t,k}(x)}) $$
 Define $\script{R}$ as the set of all vectors in $\mathbb{R}^{k+1}$ that are entrywise greater than or equal to $h(x)$ for some $x \in \script{X}$: 
$$ \script{R} = \{a\in \mathbb{R}^{k+1}: a \geq h(x) \:\: \mbox{ for some $x \in \script{X}$}  \} $$
It can be shown that $h$ is a continuous function, each of the $k+1$ components 
of $h$ is a convex function over $x \in \script{X}$, and $\script{R}$ is a closed and convex set (see Appendix \ref{appendix:structural}).

\begin{lem} \label{lem:causal} (Valid decisions) Let $\{X_t\}_{t=0}^{\infty}$ represent a valid decision sequence. Then for all $t \in \{0, 1, 2, \ldots\}$: 
$$(\expect{f_t(X_t)}, \expect{g_{t,1}(X_t)}, \ldots, \expect{g_{t,k}(X_t)}) \in  \script{R} $$
Hence, for all  integers $T>0$: 
$$\frac{1}{T}\sum_{t=0}^{T-1}(\expect{f_t(X_t)}, \expect{g_{t,1}(X_t)}, \ldots, \expect{g_{t,k}(X_t)}) \in \script{R} $$
\end{lem} 

\begin{proof} 
See Appendix \ref{appendix:causal}.
\end{proof} 

\begin{lem} \label{lem:optimality} (Optimality)  If the problem \eqref{eq:p1}-\eqref{eq:p3} is feasible, then there is a deterministic
vector $x^* \in \script{X}$ that satisfies the following for all slots $t \in \{0, 1, 2, \ldots\}$: 
\begin{align*}
\expect{f_t(x^*)} &= f^* \\
\expect{g_{t,i}(x^*)} &\leq 0 \quad , \forall i \in \{1, \ldots, k\}
\end{align*}
In particular, the set $\tilde{\script{A}}$ is nonempty and $x^* \in \tilde{\script{A}}$. 
Further, the vector $(f^*, 0, \ldots, 0)$ is on the boundary of the 
set $\script{R}$. 
\end{lem} 

\begin{proof} 
See Appendix \ref{appendix:optimality}. 
\end{proof} 

Since $(f^*, 0, \ldots, 0)$ is on the boundary of the convex set 
$\script{R}$, the \emph{hyperplane separation theorem} ensures there are nonnegative values $\gamma_0, \gamma_1, \ldots, \gamma_k$ such that $\sum_{i=0}^{k} \gamma_i a_i \geq \gamma_0 f^*$   for all $(a_0, \ldots, a_k) \in \script{R} $.   The special case when $\gamma_0> 0$ is called a \emph{nonvertical supporting hyperplane} \cite{bertsekas-convex}.  The following assumption is equivalent to the existence of a nonvertical supporting  hyperplane.\footnote{If $\gamma_0>0$ then Assumption \ref{assumption:lagrange-multipliers} holds by defining $\mu_i = \gamma_i/\gamma_0$ for $i \in \{1, \ldots, k\}$.}  

\begin{assumption} \label{assumption:lagrange-multipliers} (Existence of Lagrange multipliers)  There are nonnegative values $\mu_1, \ldots, \mu_k$, called \emph{Lagrange multipliers}, such that for any valid decision sequence $\{X_t\}_{t=0}^{\infty}$ and any slot $t \in\{0, 1,2, \ldots\}$ we have: 
\begin{equation} \label{eq:lagrange-assumption} 
 \expect{f_t(X_t)} + \sum_{i=1}^k \mu_i \expect{g_{t,i}(X_t)} \geq f^* 
 \end{equation} 
\end{assumption} 

Assumption \ref{assumption:lagrange-multipliers} is mild and shall be used to replace the more stringent Assumption \ref{assumption:slater}.

\subsection{Queue bound for Model 2}

\begin{thm} \label{thm:q-bound} Suppose Model 2 holds, problem \eqref{eq:p1}-\eqref{eq:p3}
is feasible, and the Lagrange multiplier assumption (Assumption \ref{assumption:lagrange-multipliers}) holds for a nonnegative vector $\mu=(\mu_1, \ldots, \mu_k)$.  Then for all integers $T \geq 1$ we have: 
$$ \expect{\norm{Q(T+1)}} \leq    2V\norm{\mu} + \sqrt{2CT + \frac{TV^2G^2}{\alpha} + 2\alpha D^2 + \frac{2TV^2G^2(\sum_{i=1}^{k}\mu_i)^2}{\alpha}} $$
where constants $C, G, D$ are defined in \eqref{eq:C}, \eqref{eq:G}, \eqref{eq:D}.  In particular, if we fix $\epsilon>0$ and define
$V=1/\epsilon$, $\alpha = 1/\epsilon^2$, then for all $T\geq 1/\epsilon^2$ we have: 
\begin{equation} \label{eq:theorem2} 
\frac{\expect{\norm{Q(T+1)}}}{T} \leq O(\epsilon)
\end{equation} 
\end{thm} 

\begin{proof} 
Fix $t \in \{1, 2, 3, \ldots\}$.
Since the problem is feasible, Lemma \ref{lem:optimality} ensures the 
set $\tilde{\script{A}}$ is nonempty and 
there is an optimal fixed vector $x^* \in \script{X}$. 
  Hence, the result of Lemma \ref{lem:iid-lemma} holds. 
Substituting $x=x^*$ into the right-hand-side of \eqref{eq:iid-drift-plus-penalty} gives: 
 \begin{align*}
&\expect{\Delta(t)} +  \frac{\alpha}{2} \expect{\norm{X_t-X_{t-1}}^2}  \nonumber \\
& \quad \leq C + V\underbrace{\expect{f_{t-1}(x^*)}}_{f^*}- V\expect{f_{t-1}(X_{t-1})}  + \alpha \expect{\norm{x^*-X_{t-1}}^2}- \alpha\expect{\norm{x^*-X_t}^2} + \frac{V^2G^2}{2\alpha}  
\end{align*}
Substituting \eqref{eq:lagrange-assumption} into the right-hand-side of the above inequality gives: 
\begin{align*}
&\expect{\Delta(t)} + \frac{\alpha}{2}\expect{\norm{X_t-X_{t-1}}^2} \\
&\leq C + \frac{V^2G^2}{2\alpha} + V\sum_{i=1}^k\mu_i \expect{g_{t-1,i}(X_{t-1})} 
+ \alpha\expect{\norm{x^*-X_{t-1}}^2} - \alpha\expect{\norm{x^*-X_t}^2} 
\end{align*}
Fix $T>0$. Summing over $t \in \{1, \ldots, T\}$, dividing by $T$, and using $\expect{L(1)}=0$ gives: 
\begin{align*}
&\frac{\expect{L(T+1)} - \expect{L(1)}}{T} + \frac{\alpha}{2T}\sum_{t=1}^{T} \expect{\norm{X_t-X_{t-1}}^2} \nonumber \\
&\leq C + \frac{V^2G^2}{2\alpha} + V\sum_{i=1}^k\mu_i\left[\frac{1}{T}\sum_{t=0}^{T-1}\expect{g_{t,i}(X_t)}\right] + \frac{\alpha \expect{\norm{x^*-X_0}^2}- \expect{\norm{x^*-X_{T}}^2}}{T} \nonumber \\
&\leq C + \frac{V^2G^2}{2\alpha} +  \frac{\alpha D^2}{T} + V\sum_{i=1}^k\mu_i\left[\frac{\expect{Q_i(T+1)}}{T} + \frac{G^2}{4\beta} + \frac{\beta}{T}\sum_{t=1}^T\expect{\norm{X_t-X_{t-1}}^2} \right] 
\end{align*}
where the final inequality holds by \eqref{eq:vq} and holds for all real numbers $\beta >0$.   Rearranging terms in the above inequality and using $\expect{L(1)}=0$ gives: 
\begin{align}
&\frac{\expect{L(T+1)}}{T} + \frac{(\alpha/2 - V\beta\sum_{i=1}^k\mu_i)}{T}\sum_{t=1}^{T} \expect{\norm{X_t-X_{t-1}}^2} \nonumber \\
&\leq C + \frac{V^2G^2}{2\alpha} + \frac{\alpha D^2}{T} + \frac{VG^2\sum_{i=1}^k\mu_i}{4\beta} + \frac{V}{T}\norm{\mu}\expect{\norm{Q(T+1)}} \label{eq:easy-to-see} 
\end{align}
where we have used the Cauchy-Schwartz inequality $\sum_{i=1}^k\mu_i Q_i(T+1) \leq \norm{\mu}\cdot\norm{Q(T+1)}$. 
Recall that $\mu_i\geq 0$ for all $i \in \{1, \ldots, k\}$. Temporarily assume that $\mu_i>0$ for at least one $i \in \{1, \ldots, k\}$, and  define: 
$$ \beta = \frac{\alpha}{4V\sum_{i=1}^k\mu_i}$$
Substituting this value of $\beta$ into \eqref{eq:easy-to-see} gives
\begin{align*}
&\frac{\expect{L(T+1)}}{T}   + \frac{\alpha}{4T}\sum_{t=1}^T\expect{\norm{X_t-X_{t-1}}^2} \nonumber \\
&\leq C + \frac{V^2G^2}{2\alpha} +  \frac{\alpha D^2}{T}  + \frac{V^2G^2(\sum_{i=1}^k \mu_i)^2}{\alpha} + \frac{V\norm{\mu}}{T}\expect{\norm{Q(T+1)}}
\end{align*}
It is easy to see that this inequality also holds in the special case $\mu_i=0$ for all $i \in \{1, \ldots, k\}$, since then 
the final two terms on the right-hand-side of \eqref{eq:easy-to-see} disappear.  Multiplying the above inequality by $2T$ and 
using the definition $L(T+1) = \frac{1}{2}\norm{Q(T+1)}^2$ 
gives: 
\begin{align}
 &\expect{\norm{Q(T+1)}^2} + \frac{\alpha}{2}\sum_{t=1}^T\expect{\norm{X_t-X_{t-1}}^2} \nonumber \\
 &\leq  2CT + \frac{TV^2G^2}{\alpha} + 2\alpha D^2 + \frac{2TV^2G^2(\sum_{i=1}^k\mu_i)^2}{\alpha} + 
 2V\norm{\mu}\expect{\norm{Q(T+1)}} \label{eq:reuse-bound} 
 \end{align}
Define $z = \expect{\norm{Q(T+1)}}$ and note that $z^2 \leq \expect{\norm{Q(T+1)}^2}$. 
Inequality  \eqref{eq:reuse-bound} 
implies 
$$ z^2  \leq w z+ y $$
for the nonnegative quantities $w, y$ defined: 
\begin{align*}
w &= 2V\norm{\mu} \\
y &= 2CT + \frac{TV^2G^2}{\alpha} + 2\alpha D^2 + \frac{2TV^2G^2(\sum_{i=1}^k\mu_i)^2}{\alpha} 
\end{align*}
and so
$$ z \leq \frac{w + \sqrt{w^2 + 4y}}{2} \leq w + \sqrt{y} $$
where the final inequality uses the fact that $\sqrt{a+b} \leq \sqrt{a} + \sqrt{b}$ for nonnegative real numbers $a,b$. 
In particular: 
$$ \expect{\norm{Q(T+1)}} \leq   2V\norm{\mu} + \sqrt{y}$$
This proves the result. 
\end{proof} 

\subsection{Performance bound for Model 2} 

\begin{thm} \label{thm:constraint-bound-model-2} Suppose Model 2 holds, the problem \eqref{eq:p1}-\eqref{eq:p3}
is feasible, and the Lagrange multiplier assumption (Assumption \ref{assumption:lagrange-multipliers}) holds for a nonnegative vector $\mu=(\mu_1, \ldots, \mu_k)$.  Fix $\epsilon>0$ and define $V= 1/\epsilon$, $\alpha = 1/\epsilon^2$.  Then for all $T \geq 1/\epsilon^2$ and all $i \in \{1, \ldots, k\}$ we have: 
\begin{align}
\frac{1}{T}\sum_{t=0}^{T-1}\expect{f_t(X_t)} &\leq f^* + O(\epsilon) \label{eq:iid-f-model2} \\
\frac{1}{T}\sum_{t=0}^{T-1}\expect{g_{t,i}(X_t)} &\leq O(\epsilon) \label{eq:iid-g-model2} 
\end{align}
\end{thm} 

\begin{proof} 
Inequality \eqref{eq:iid-f-model2} follows directly from Theorem \ref{thm:model-1-performance} part (a) by using $x=x^*$ and 
noting that $\expect{f_t(x^*)}=f^*$.\footnote{Strictly 
speaking, Theorem \ref{thm:model-1-performance} is stated assuming that $V$ is a positive integer, while here we simply 
assume $V$ is positive.  The assumption that $V$ is a positive integer was only needed for part (b) of Theorem \ref{thm:model-1-performance},
and hence the statement in the current theorem is correct.}
To prove \eqref{eq:iid-g-model2}, 
fix $i \in \{1, \ldots, k\}$ and $V=1/\epsilon$.  Taking expectations of \eqref{eq:vq} and using $\beta = 1/\epsilon$ gives: 
\begin{align*}
\frac{1}{T}\sum_{t=0}^{T-1}\expect{g_{t,i}(X_t)} &\leq \frac{\expect{Q_i(T+1)}}{T} + \frac{G^2\epsilon}{4} + \frac{1/\epsilon}{T}\sum_{t=1}^T\expect{\norm{X_t-X_{t-1}}^2}   \\
&\leq O(\epsilon) + \frac{1}{\epsilon T} \sum_{t=1}^T\expect{\norm{X_t-X_{t-1}}^2}  
\end{align*}
where the final inequality holds by Theorem \ref{thm:q-bound} (specifically, by \eqref{eq:theorem2}). It remains to show that the final term on the right-hand-side of the above  inequality is $O(\epsilon)$. 

Multiplying \eqref{eq:reuse-bound} by $\frac{2}{\alpha \epsilon T}$ and neglecting the nonnegative 
term $\expect{\norm{Q(T+1)}^2}$ gives: 
\begin{align*}
\frac{1}{\epsilon T}\sum_{t=1}^T\expect{\norm{X_t-X_{t-1}}^2} &\leq \frac{4C}{\alpha \epsilon} + \frac{2V^2G^2}{\alpha^2\epsilon} + \frac{4D^2}{\epsilon T} + 
\frac{4V^2G^2(\sum_{i=1}^k\mu_i)^2}{\alpha^2 \epsilon} + \frac{4V\norm{\mu}}{\alpha \epsilon}\frac{\expect{\norm{Q(T+1)}}}{T}\\
&=4C\epsilon + 2G^2\epsilon + \frac{4D^2}{\epsilon T} + 4G^2(\sum_{i=1}^k\mu_i)^2\epsilon + 4\norm{\mu}\frac{\expect{\norm{Q(T+1)}}}{T} 
\end{align*}
where the final equality holds by substituting $V=1/\epsilon, \alpha=1/\epsilon^2$. 
The right-hand-side of the above bound is indeed $O(\epsilon)$ whenever $T\geq 1/\epsilon^2$ (recall \eqref{eq:theorem2}). 
\end{proof}

\appendix 

\subsection{Properties of $h$ and $\script{R}$} \label{appendix:structural}

Let $x, y \in \script{X}$.  Then for each $\omega \in \Omega$ we have by the bounded subgradient assumption: 
\begin{equation} \label{eq:cont} 
 |\hat{f}(x,\omega) - \hat{f}(y, \omega)| \leq G \norm{x-y} 
 \end{equation} 
 where $G$ is defined in \eqref{eq:G}. 
Hence, 
\begin{align*}
|\expect{f_t(x)} - \expect{f_t(y)}| &\overset{(a)}{\leq} \expect{|f_t(x) - f_t(y)|} \\
&= \expect{|\hat{f}(x,\omega_t)- \hat{f}(y,\omega_t)|} \\
&\overset{(b)}{\leq} \expect{ G\norm{x-y}} \\
&= G \norm{x-y} 
\end{align*}
where (a) holds by Jensen's inequality applied to the absolute value function; (b) holds by \eqref{eq:cont}.  Similarly, it holds for all $i \in \{1, \ldots, k\}$ that: 
$$ |\expect{g_{t,i}(x)} - \expect{g_{t,i}(y)}| \leq G\norm{x-y} $$
Define the vector-valued function $h:\script{X}\rightarrow\mathbb{R}^{k+1}$ by 
$$ h(x) = (\expect{f_t(x)} , \expect{g_{t,1}(x)}, \ldots, \expect{g_{t,k}(x)}) $$
It follows that $h(x)$ is a continuous function defined over a compact set.  Its image $h(\script{X})$ is thus compact.
Since $\script{R}$ is the set of all vectors entrywise greater than or equal to some vector in $h(\script{X})$, set $\script{R}$ is closed. 
 It can be shown that each component of the $h(x)$ function is convex over $x \in \script{X}$. The proof that $\script{R}$ is convex follows directly and is omitted for brevity.

\subsection{Proof of Lemma  \ref{lem:causal}} \label{appendix:causal}

Let $\{X_t\}_{t=0}^{\infty}$ be a valid decision sequence.   In particular, 
for all $t \in \{0, 1, 2, \ldots\}$ we know $X_t \in \script{X}$ and $X_t$ is independent of $\omega_t$.  
Fix $t \in \{0, 1, 2, \ldots\}$ and define the deterministic vector $x_t = \mathbb{E}_{X_t}[X_t]$. Since
$X_t$ is a random vector in the compact and convex set $\script{X}$, its expectation $x_t$ must also
be in $\script{X}$. Since $X_t$ and $\omega_t$ are independent we have 
$\mathbb{E}_{X_t|\omega_t}[X_t |\omega_t]= x_t$. Thus, 
\begin{align*}
\expect{f_t(X_t)} &= \expect{\hat{f}(X_t,\omega_t)} \\
&= \mathbb{E}_{\omega_t}\left[\mathbb{E}_{X_t|\omega_t}\left[\hat{f}(X_t, \omega_t)|\omega_t \right]\right]\\
&\geq  \mathbb{E}_{\omega_t}\left[\hat{f}(\mathbb{E}_{X_t|\omega_t}[X_t|\omega_t], \omega_t)\right]\\
&= \mathbb{E}_{\omega}\left[\hat{f}(x_t, \omega_t)\right]\\
&= \expect{f_t(x_t)} 
\end{align*}
where the inequality is due to Jensen's inequality for the function $\hat{f}(x,\omega)$ which is convex over $x \in \script{X}$ for each fixed $\omega_t \in \Omega$.  Similarly, 
$$ \expect{g_{t,i}(X_t)} \geq  \expect{g_{t,i}(x_t)} \quad , \forall i  \in \{1, \ldots, k\}  $$
It follows that: 
\begin{align*}
h(x_t) &= (\expect{f_t(x_t)}, \expect{g_{t,1}(x_t)}, \ldots, \expect{g_{t,k}(x_t)}) \\
&\leq (\expect{f_t(X_t)},  \expect{g_{t,1}(X_t)}, \ldots, \expect{g_{t,k}(X_t)})
\end{align*}
and so: 
$$(\expect{f_t(X_t)},  \expect{g_{t,1}(X_t)}, \ldots, \expect{g_{t,k}(X_t)})\in \script{R}$$
 This proves the first part of the lemma.  Since $\script{R}$ is a convex set, any convex combination of points
 in $\script{R}$ is also in $\script{R}$, and so for any integer $T\geq 1$ we have: 
 \begin{equation} \label{eq:causal-appendix} 
 \frac{1}{T}\sum_{t=0}^{T-1}(\expect{f_t(X_t)},  \expect{g_{t,1}(X_t)}, \ldots, \expect{g_{t,k}(X_t)})\in \script{R}
 \end{equation} 
This completes the proof of Lemma \ref{lem:causal}. 

\subsection{Proof of Lemma \ref{lem:optimality}} \label{appendix:optimality}

Fix $\epsilon>0$.  Let $\{X_t\}_{t=0}^{\infty}$ be a valid decision sequence that satisfies the 
desired constraints \eqref{eq:p2}-\eqref{eq:p3} and that achieves an objective value within $\epsilon/2$ of optimality: 
\begin{align*}
 \limsup_{T\rightarrow\infty}\frac{1}{T}\sum_{t=0}^{T-1} \expect{f_t(X_t)} &\leq f^* + \epsilon/2 \\
 \limsup_{T\rightarrow\infty} \frac{1}{T}\sum_{t=0}^{T-1} \expect{g_{t,i}(X_t)} &\leq 0 \quad , \forall i \in \{1, \ldots, k\} 
 \end{align*}
 In particular, there is a positive integer $T$ such that: 
 \begin{align*}
\frac{1}{T}\sum_{t=0}^{T-1} \expect{f_t(X_t)} &\leq f^* + \epsilon \\
\frac{1}{T}\sum_{t=0}^{T-1} \expect{g_{t,i}(X_t)} &\leq 0 + \epsilon \quad , \forall i \in \{1, \ldots, k\} 
 \end{align*}
 Since the vector of \eqref{eq:causal-appendix} is in the set $\script{R}$ and the vector
 $(f^*+\epsilon, \epsilon, \epsilon, \ldots, \epsilon)$ is entrywise greater than or equal to this
 vector, we know that: 
 $$  (f^*+\epsilon, \epsilon, \epsilon, \ldots, \epsilon) \in \script{R} $$
 This holds for all $\epsilon>0$. Since $\script{R}$ is closed, it follows that: 
  $$  (f^*, 0, 0, \ldots, 0) \in \script{R} $$
By definition of $\script{R}$, there must be a deterministic vector $x^* \in \script{X}$ such that: 
$$ (\expect{f_t(x^*)}, \expect{g_{t,1}(x^*)}, \ldots, \expect{g_{t,k}(x^*)})  \leq (f^*, 0, 0, \ldots, 0) $$ 
  Now if $\expect{f_t(x^*)}<f^*$, then  the (valid) 
 decisions $X_t = x^*$ for all $t$    would satisfy all constraints and reach an objective value strictly less than $f^*$ (contradicting the fact that $f^*$ is the optimal objective value).  Hence, $\expect{f_t(x^*)}=f^*$. This proves the first part of the lemma. 
 
 For the second part, we already know that $(f^*, 0, 0, \ldots, 0) \in \script{R}$. To show this is on the \emph{boundary} of $\script{R}$, just note that for all $\delta>0$ the point 
 $(f^*-\delta, 0, \ldots,0)$ cannot be in $\script{R}$, else we could construct a valid decision sequence 
 that satisfies all desired constraints and achieves an objective value strictly smaller than $f^*$.  This completes
 the proof of Lemma \ref{lem:optimality}. 

\subsection{Convergence time to regret conversion} \label{appendix:regret} 

This subsection shows how to use a doubling trick to convert between the convergence
time bound and the regret bound. The doubling trick is standard and can be used in 
different contexts, see, for example,  \cite{Shalev-Shwartz11FoundationTrends}. 

For simplicity we consider algorithms with deterministic guarantees.
Expectation guarantees can be treated similarly. 
Consider a system with functions $\{f_t\}_{t=0}^{\infty}$ and $\{g_{t,i}\}_{t=0}^{\infty}$ for  
$i \in \{1, \ldots, k\}$.  
 Suppose there are constants $c$ and $d$, a set $\script{B} \subseteq \script{X}$, 
together with an algorithm parameterized
by $\epsilon$ such 
that for all $\epsilon>0$ the algorithm can be configured to ensure: 
\begin{align}
\frac{1}{T}\sum_{t=0}^{T-1} f_t(X_t) &\leq \frac{1}{T}\sum_{t=0}^{T-1} f_t(x) + c\epsilon + \frac{d}{\epsilon T} \quad, \forall x \in \script{B} , \forall T>0 \label{eq:convergence1}\\
\frac{1}{T}\sum_{t=0}^{T-1} g_{t,i}(X_t) &\leq c\epsilon + \frac{d}{\epsilon T} \quad , \forall i \in \{1, \ldots, k\}, \forall T>0 \label{eq:convergence2} 
\end{align} 
The algorithm of the current paper indeed ensures such performance 
(Theorems \ref{thm:performance-bound} and \ref{thm:constraint-bound}). 

Now consider a modified algorithm implemented over 
successive frames with sizes $\{T_1, T_2, T_3, \ldots\}$ such that 
$T_m = 2^m$ for $m \in \{1, 2, 3, \ldots\}$. In each frame $m \in \{1, 2, 3, ...\}$, 
restart the algorithm and use $\epsilon_m = 1/\sqrt{T_m} =(\frac{1}{\sqrt{2}})^m$.  We want to show that this algorithm ensures an $O(\sqrt{T})$ regret for both the objective function and constraints, for all $T>0$.  Define $\script{T}_m$ as the set of all integer 
times $t$ within frame $m$. 

Fix $T>3$ and let $M>0$ be the integer such that 
$$\underbrace{T_1 + ... + T_M}_{2^{M+1}-1}  <  T  \leq T_1 + ... + T_{M+1}$$ 
Fix $x \in \script{B}$.  Then under this modified algorithm with successively doubled frame sizes, we have for each frame $j \in \{1, ..., M\}$: 
\begin{align}
\frac{1}{T_j}\sum_{t \in \script{T}_j} f_t(X_t) &\leq \frac{1}{T_j}\sum_{t \in \script{T}_j} f_t(x) + c\epsilon_j + \frac{d}{\sqrt{T_j}} \quad, \forall x \in \script{B} , \forall T>0 \label{eq:convergence11}\\
\frac{1}{T_j}\sum_{t\in \script{T}_j}  g_{t,i}(X_t) &\leq c\epsilon_j + \frac{d}{\sqrt{T_j}} \quad , \forall i \in \{1, \ldots, k\}, \forall T>0 \label{eq:convergence22} 
\end{align} 
where the above uses $T_j \epsilon_j = \sqrt{T_j}$.  Define $\theta= T-(T_1+...+T_M)$ and note that $0 \leq \theta \leq 2^{M+1}$. 
We have: 
\begin{equation} \label{eq:for-case-2} 
\frac{1}{\theta} \sum_{t=T_1+...+T_m}^{T-1} f_t(X_t) \leq \frac{1}{\theta}  \sum_{t=T_1+...+T_M}^{T-1} f_t(x) + c\epsilon_{M+1} + \frac{d}{\epsilon_{M+1} \theta} 
\end{equation} 
Multiplying \eqref{eq:for-case-2} by $\theta$ gives: 
\begin{align}
\sum_{t=T_1+...+T_m}^{T-1} f_t(X_t) &\leq  \sum_{t=T_1+...+T_m}^{T-1} f_t(x) + \theta c \epsilon_{M+1} + \frac{d}{\epsilon_{M+1}} \nonumber \\
&\overset{(a)}{\leq} \sum_{t=T_1+...+T_m}^{T-1} f_t(x) + c2^{M+1}(\frac{1}{\sqrt{2}})^{M+1} + d\sqrt{2^{M+1}} \nonumber \\
&\overset{(b)}{=} \sum_{t=T_1+...+T_m}^{T-1} f_t(x)  + (c+d)\sqrt{T} \label{eq:residue} 
\end{align} 
where (a) uses $\theta \leq 2^{M+1}$ and (b) uses $2^{M+1}\leq T$. 
Multiplying \eqref{eq:convergence22} by $T_j$ and using $\epsilon_j T_j = \sqrt{T_j}$ gives, for each $j \in \{1, \ldots, M\}$:  
$$ \sum_{t\in \script{T}_j}  g_{t,i}(X_t) \leq  (c+d) \sqrt{T_j} $$
Summing the above inequality over all $j \in \{1, \ldots, M\}$ with \eqref{eq:residue} gives: 
\begin{align*}
\sum_{t=0}^{T-1} f_t(X_t) &\leq \sum_{t=0}^{T-1} f_t(x) + (c+d)\sum_{j=1}^M \sqrt{T_j} + (c+d)\sqrt{T}\\
&=\sum_{t=0}^{T-1} f_t(x) + (c+d)\sqrt{T} + (c+d)\sum_{j=1}^M (\sqrt{2})^j \\
&= \sum_{t=0}^{T-1} f_t(x) + (c+d)\sqrt{T} + (c+d)\sqrt{2}\frac{(\sqrt{2})^{M}-1}{\sqrt{2}-1}\\
&\leq \sum_{t=0}^{T-1} f_t(x) + (c+d)\sqrt{T} + \frac{(c+d)\sqrt{2}}{\sqrt{2}-1}\sqrt{2^M} \\
&\leq  \sum_{t=0}^{T-1} f_t(x) + (c+d)\sqrt{T} + \frac{(c+d)\sqrt{2}}{\sqrt{2}-1} \sqrt{T}
\end{align*} 
and so the regret is $\beta \sqrt{T}$ for $\beta = (c+d)\left[1 + \frac{\sqrt{2}}{\sqrt{2}-1}\right]$.   A similar $O(\sqrt{T})$ regret holds for the constraints (proof omitted for brevity). 

\bibliographystyle{unsrt}
\bibliography{../../../latex-mit/bibliography/refs}
\end{document}